\documentclass{amsart}

\usepackage{amsfonts}
\usepackage{bbm}
\usepackage{amsmath}
\usepackage{amssymb}
\usepackage[english]{babel}
\usepackage{amsthm}
\usepackage{tikz}
\usepackage{textgreek}
\usetikzlibrary{arrows.meta, calc, patterns}
\newtheorem {thm}{Theorem}

\newtheorem {prop}[thm]{Proposition}
\newtheorem*{conj*}{Conjecture}

\theoremstyle{definition}

\theoremstyle{remark}
\newtheorem*{rem*}{Remark}

\DeclareMathOperator{\dist}{dist}
\DeclareMathOperator{\inter}{int}
\DeclareMathOperator{\conv}{conv}

\begin{document}

\title[A proof of a conjecture by Haviv, Lyubashevsky and Regev]{A proof of a conjecture by Haviv, Lyubashevsky and Regev on the second moment of a lattice Voronoi cell}
\author{Alexander Magazinov}
\address{Tel Aviv University, School of Mathematical Sciences, Ramat Aviv, Tel Aviv 69978, Israel}
\email{magazinov@post.tau.ac.il}
\thanks{Supported in part by ERC Starting Grant 678520.}

\begin{abstract}
In this short note we prove a sharp lower bound for the second moment of a lattice Voronoi cell in terms of the respective covering radius. This gives an affirmative answer to a conjecture by Haviv, Lyubashevsky and Regev. We also characterize those lattice Voronoi cells for which this lower bound is attained.
\end{abstract}

\maketitle

\section{Introduction}

Consider the $n$-dimensiomal space $\mathbb R^n$. Denote by $\| x \|$ the standard Euclidean norm of a vector $x \in \mathbb R^n$,
and by $|X|$ --- the $n$-dimensional Lebesgue measure of a set $X \subset \mathbb R^n$. The notation $\dist(\cdot, \cdot)$ will refer
to the Euclidean distance between two sets or between a point and a set.

If $\Lambda \subset \mathbb R^n$ is an $n$-dimensional lattice then the quantity
\begin{equation*}
  R = R(\Lambda) = \sup\limits_{x \in \mathbb R^n} \dist(x, \Lambda)
\end{equation*}
is called the covering radius of $\Lambda$. If $v \in \Lambda$, define the {\it Voronoi cell} of $v$ with
respect to $\Lambda$ as follows:
\begin{equation*}
  V_{\Lambda}(v) = \{ x \in \mathbb R^n : \| x - v \| = \dist(x, \Lambda) \}.
\end{equation*}
In other words, the Voronoi cell $V_{\Lambda}(v)$ consists of all points $x \in \mathbb R^n$ that are at least as close
to $v$ as to any other point of $\Lambda$. A Voronoi cell is known to be a convex polytope.

It is clear that the covering radius $R(\Lambda)$ is connected to the notion of a Voronoi cell through the relation
\begin{equation*}
  R(\Lambda) = \sup\limits_{x \in V_{\Lambda}(v)} \| x - v \| \quad \text{(for any $v \in \Lambda$).}
\end{equation*}

The main result of this paper is Theorem~\ref{thm:main} below, providing an affirmative answer to the conjecture by
Haviv, Lyubashevsky and Regev~\cite[Conjecture 1.3]{Reg} (the {\it HLR Conjecture}, for brevity).

\begin{thm}\label{thm:main}
Let $\Lambda \subset \mathbb R^n$ be an $n$-dimensional lattice containing the origin $\mathbf 0$ and let $P = V_{\Lambda}(\mathbf 0)$.
If $R$ is the covering radius of $\Lambda$ then the following inequality holds:
\begin{equation}\label{eq:main}
\int\limits_{P} \| x \|^2 dx \geq \frac{R^2}{3} \cdot |P|.
\end{equation}
\end{thm}

The initial motivation for the HLR Conjecture provided in~\cite{Reg} comes from considering the Covering Radius
Problem ({\sf CRP}) in computational complexity. The {\sf CRP} with approximation factor $\gamma \geq 1$ is the problem of distinguishing between {\sf YES} instances, which are lattices with covering radius at most $r$, and {\sf NO} instances, which are lattices with covering radius bigger than $\gamma r$. Speaking informally, one aims to design a simple
protocol by which a prover can convince a (randomized) verifier that an
instance of {\sf CRP} is a {\sf YES} instance. If this is possible, one concludes that the {\sf CRP} with approximation factor $\gamma$ belongs to the so-called {\sf AM} class of complexity, which is, apparently, not much
wider than {\sf NP}. \cite{Reg} shows that {\sf CRP} with factor $\gamma$ is in {\sf AM} for any $\gamma > \sqrt{3}$ provided that the HLR Conjecture is true. For details, see~\cite{Reg} and the references therein.

Another motivation was
explained to the author by Barak Weiss, from whom the author learned about the HLR Conjecture. It is connected with the famous Minkowski conjecture,
which, in one of its equivalent formulations, reads as follows.

\begin{conj*}[Minkowski; see, for instance,~\cite{Wds}]
Let $(e_1, e_2, \ldots, e_n)$ be an orthonormal basis in $\mathbb R^n$ with respect to the Euclidean scalar product $\langle \cdot, \cdot \rangle$.
Let $\Lambda \subset \mathbb R^n$ be an $n$-dimensional lattice, $\mathbf 0 \in \Lambda$. Assume that $\Lambda$ has unit covolume, i.e, $|V_{\Lambda}(\mathbf 0)| = 1$.
Then for every vector $t \in \mathbb R^n$ there exists a point $v \in \Lambda + t$ such that
\begin{equation*}
  \left| \prod\limits_{i = 1}^n \langle v, e_i \rangle \right| \leq \left( \frac{1}{2} \right)^n.
\end{equation*}
\end{conj*}
The conjecture is commonly attributed to Minkowski, however, the author is not aware of any reference in Minkowski's work. The recent state of the 
conjecture is reflected in~\cite{SW}.

By means of the HLR Conjecture, the paper~\cite{RD} relates the Minkowski Conjecture to an another notable open problem, the Strong Slicing Conjecture~\cite[Section~2]{Mec}, stated below.

\begin{conj*}[Strong Slicing Conjecture (for symmetric bodies)]
Let $(e_1, e_2, \ldots, e_n)$ be an orthonormal basis in $\mathbb R^n$ with respect to the Euclidean scalar product $\langle \cdot, \cdot \rangle$.
Let $K \subset \mathbb R^n$ be a $\mathbf 0$-symmetric convex body of unit volume. Denote
\begin{align*}
  a_{ij}(K) = \int\limits_{K} \langle x, e_i \rangle \cdot \langle x, e_j \rangle \, dx. & \qquad \text{and} \\
  L_K =  \bigl( \det(a_{ij}(K))_{i, j = 1}^n \bigr)^{\frac{1}{2n}}.
\end{align*}
Then $L_K \leq \frac{1}{\sqrt{12}}$. (The equality is achieved for the unit cube and its affine images.)
\end{conj*}

In its weaker version the Slicing Conjecture asserts that, if $K$ is a $\mathbf 0$-symmetric convex body of unit volume, then its isotropic
constant $L_K$ is bounded from above by a universal constant. In particular, the upper bound should be independent of the dimension. The notion of the isotropic constant is extremely important in convex geometry; its significance is justified by numerous applications (see, for 
instance,~\cite{AGM}).

An argument in~\cite[Section~6]{RD} shows that if the HLR Conjecture and the Strong Slicing Conjecture are true, then the Minkowski Conjecture is true as well.

We will also give an explicit answer when the inequality~\eqref{eq:main}
in the HLR Conjecture turns into an equality.

\section{Proof of Theorem 1}

Let $t \in \mathbb R^n$ be any vector satisfying
\begin{equation*}
\| t \| = R = \dist(t, \Lambda).
\end{equation*}
Consider the collection of polytopes
\begin{equation*}
\mathcal T(t) = \{ P + t + v : v \in \Lambda \}
\end{equation*}
Since $\mathcal T(t)$ is a tessellation of $\mathbb R^n$, we have
\begin{equation}\label{eq:1}
\int\limits_{P} \| x \|^2 dx = \sum\limits_{v \in \Lambda} \left( \int\limits_{P \cap (P + t + v)}  \| x \|^2 dx \right).
\end{equation}

Consider a single summand in the right-hand side of~\eqref{eq:1}. Denote
\begin{equation*}
Q(t, v) = P \cap (P + t + v).
\end{equation*}
The set $Q(t, v)$ has a center of symmetry at the point $\frac{t + v}{2}$, because the polytopes $P$ and $P + t + v$
are symmetric to each other with respect to the point $\frac{t + v}{2}$.  Consequently,
\begin{multline}\label{eq:2}
\int\limits_{Q(t, v)}  \| x \|^2 dx = 
\int\limits_{Q(t, v)} \left(
\left \| \frac{t + v}{2} \right\|^2 + 
\left \| x - \frac{t + v}{2} \right\|^2 +
2 \left\langle \frac{t + v}{2}, x - \frac{t + v}{2} \right\rangle
\right) \, dx
\\
= \left\| \frac{t + v}{2} \right\|^2 |Q(t, v)| + \int\limits_{Q(t, v)} \left\| x - \frac{t + v}{2} \right\|^2 dx.
\end{multline}
Indeed, by the symmetry of $Q(t, v)$, the term 
$\left\langle \frac{t + v}{2}, x - \frac{t + v}{2} \right\rangle$ vanishes
after integration.

Let us notice that
\begin{equation}\label{eq:3}
\left\| x - \frac{t + v}{2} \right\| \geq \dist \left( x, \frac{1}{2} \Lambda + \frac{t}{2} \right).
\end{equation}
Therefore, inserting~\eqref{eq:2} and~\eqref{eq:3} into~\eqref{eq:1}, we have
\begin{multline}\label{eq:4}
\int\limits_{P} \| x \|^2 dx =
\sum\limits_{v \in \Lambda} \left\| \frac{t + v}{2} \right\|^2 |Q(t, v)| + \sum\limits_{v \in \Lambda} \int\limits_{Q(t, v)} \left\| x - \frac{t + v}{2} \right\|^2 dx \geq \\
\sum\limits_{v \in \Lambda} \left\| \frac{t + v}{2} \right\|^2 |Q(t, v)| + \sum\limits_{v \in \Lambda} \int\limits_{Q(t, v)} \dist \left( x, \frac{1}{2} \Lambda + \frac{t}{2} \right)^2 dx \geq \\
|P| \cdot \inf\limits_{v \in \Lambda} \left\| \frac{t + v}{2} \right\|^2 + \int\limits_{P} \dist \left( x, \frac{1}{2} \Lambda + \frac{t}{2} \right)^2 dx.
\end{multline}

One can see that
\begin{equation}\label{eq:5}
\inf\limits_{v \in \Lambda} \left\| \frac{t + v}{2} \right\| = \frac{1}{2}\inf\limits_{v \in \Lambda} \| t - v \| = \frac{\dist(t, \Lambda)}{2} = \frac{\| t \|}{2}.
\end{equation}

Let $w_1, w_2, \ldots, w_{2^n} \in \frac{1}{2} \Lambda$ be a $2^n$-tuple of points, pairwise incomparable modulo $\Lambda$. Then the set
\begin{equation*}
D = \bigcup\limits_{i = 1}^{2^n} \left( \frac{1}{2} P + w_i + \frac{t}{2} \right)
\end{equation*}
is a fundamental domain for $\Lambda$. Since $P$ is also a fundamental domain for $\Lambda$ and since the function $f(x) = \dist \left( x, \frac{1}{2} \Lambda + \frac{t}{2} \right)^2$
is $\Lambda$-periodic, we have
\begin{multline}\label{eq:7}
\int\limits_{P} \dist \left( x, \frac{1}{2} \Lambda + \frac{t}{2} \right)^2 dx = \sum\limits_{i = 1}^{2^n}
\left( \int\limits_{\frac{1}{2} P + w_i + \frac{t}{2}} \dist \left( x, \frac{1}{2} \Lambda + \frac{t}{2} \right)^2 dx \right) = \\
2^n \int\limits_{\frac{1}{2} P} \| x \|^2 dx = \frac{1}{4} \int\limits_{P} \| x \|^2 dx.
\end{multline}

Finally, inserting~\eqref{eq:5} and~\eqref{eq:7} into~\eqref{eq:4}, we obtain
\begin{equation*}
\int\limits_{P} \| x \|^2 dx \geq |P| \frac{\| t \|^2}{4} + \frac{1}{4} \int\limits_{P} \| x \|^2 dx.
\end{equation*}
Hence, indeed,
\begin{equation*}
\int\limits_{P} \| x \|^2 dx \geq |P| \frac{\| t \|^2}{3}. \qedhere
\end{equation*}
\qed

\section{The case of equality}

It seems to be a natural question to determine all $\mathbf 0$-symmetric lattice Voronoi cells $P$ minimizing the quantity $\frac{1}{|P|}\int\limits_{P} \left( \frac{\| x \|}{R} \right)^2 dx$,
where $R = \sup\limits_{x \in P} \| x \|$ is the covering radius of the corresponding lattice. A careful inspection of the proof of Theorem 1 allows us determine the minimizers.

\begin{thm}\label{thm:equality}
The inequality~\eqref{eq:main} in Theorem 1 turns into equality if and only if $P$ is a rectangular box (i.e. a direct Minkowski sum of $n$ pairwise orthogonal segments).
\end{thm}

Before we proceed with a proof, let us recall the notion of a lattice Delaunay cell. Given a lattice $\Lambda \subset \mathbb R^n$ and a Euclidean ball $B \subset \mathbb R^n$, we call the sphere $\partial B$
empty if $\inter B \cap \Lambda = \varnothing$. If $\partial B$ is
an empty sphere and $\partial B \cap \Lambda \neq \varnothing$, then
the convex polytope $\conv (\partial B \cap \Lambda)$ is called a {\it lattice Delaunay cell}.

We will need the following two propositions.

\begin{prop}\label{prop:dela}
Let $v_1, v_2, v_3$ be three vertices of a lattice Delaunay cell. Then
\begin{equation*}
\langle v_1 - v_3, v_2 - v_3 \rangle \geq 0.
\end{equation*}
In other words, a lattice Delaunay cell does not span obtuse-angled triangles.
\end{prop}

\begin{proof}
See the proof of~\cite[Proposition 13.2.8]{DeLa}.
\end{proof}

\begin{prop}\label{prop:dagr}
Let $\{ v_1, v_2, \ldots, v_{2^n} \} \subset \mathbb R^n$ be a set of $2^n$
pairwise distinct points such that the inequality
\begin{equation*}
\langle v_i - v_l, v_j - v_l \rangle \geq 0
\end{equation*}
holds for every $i, j, l \in \{ 1, 2, \ldots, 2^n \}$. Then 
$\{ v_1, v_2, \ldots, v_{2^n} \}$ is the vertex set of some 
$n$-dimensional rectangular box.
\end{prop}

\begin{proof}
See~\cite[Satz II.b.\textbeta]{DaGr}.
\end{proof}

Now we are ready to prove Theorem~\ref{thm:equality}.

\begin{proof}[Proof of Theorem~\ref{thm:equality}]
The ``if'' part is straightforward. Indeed, if 
\begin{equation*}
P = [-a_1, a_1] \times [-a_2, a_2] \times \ldots \times [-a_n, a_n],
\end{equation*}
then $R^2 = a_1^2 + a_2^2 + \ldots + a_n^2$, while
\begin{equation*}
\frac{1}{|P|} \int\limits_{P} \| x \|^2 dx = \sum\limits_{i = 1}^n 
\left(\frac{1}{2a_i} \int\limits_{-a_i}^{a_i} t^2 \, dt \right) = 
\frac{1}{3}(a_1^2 + a_2^2 + \ldots + a_n^2).
\end{equation*}

We proceed with the ``only if'' part. Assume $P$ is the Voronoi cell of $\mathbf 0$ with respect to
the lattice $\Lambda \subset \mathbb R^n$ such that the inequality~\eqref{eq:main} turns into equality.

We notice that~\eqref{eq:3} turns into equality exactly in one of the two cases:
\begin{eqnarray}
  \frac{t + v}{2} \notin \inter P \quad (\Longleftrightarrow |Q(t, v)| = 0), \label{eq:8} \\
  Q(t, v) \subseteq \frac{1}{2} P + \frac{t + v}{2}. \label{eq:9}
\end{eqnarray}

Let
\begin{equation*}
  \{ v_1, v_2, \ldots, v_k \} = \Lambda \cap \inter(2P - t).
\end{equation*}
Equivalently, $v_1, v_2, \ldots, v_k$ are exactly those points of 
$\Lambda$ for which~\eqref{eq:8} fails.

Assume that~\eqref{eq:9} fails with $v = v_i$ for some $i \in \{ 1, 2, \ldots, k \}$. Then the inequality~\eqref{eq:3} is strict. Thus the inequality~\eqref{eq:main} is strict, too, which contradicts our assumption on $P$. Hence~\eqref{eq:9} holds with $v = v_i$ for each
$i \in \{ 1, 2, \ldots, k \}$.

It is clear that
\begin{equation*}
  P = \bigcup\limits_{i = 1}^k Q(t, v_i).
\end{equation*}
Then
\begin{equation*}
  |P| = \sum\limits_{i = 1}^k |Q(t, v_i)| \leq \sum\limits_{i = 1}^k \left| \frac{1}{2} P + \frac{t + v_i}{2} \right| = \frac{k}{2^n} |P|.
\end{equation*}
Consequently, $k \geq 2^n$.

On the other hand, for every $1 \leq i < j \leq k$ we have 
$v_i, v_j \in \inter (2P - t)$. But $2P$ is a fundamental domain of the
lattice $2\Lambda$. Therefore $v_i \not \equiv v_j \pmod{2 \Lambda}$.
Since $| \Lambda / 2\Lambda | = 2^n$, we conclude that $k \leq 2^n$.

The above implies $k = 2^n$ and $Q(t, v_i) = \frac{1}{2} P + \frac{t + v_i}{2}$ for $i \in \{ 1, 2, \ldots, 2^n \}$.

Consider the positive homothety $H_i$ (with coefficient $+ \frac{1}{2}$) that sends $P$ to $\frac{1}{2} P + \frac{t + v_i}{2}$.
Since the center of $H_i$ is the point $t + v_i$ and since
\begin{equation*}
  \frac{1}{2} P + \frac{t + v_i}{2} = Q(t, v_i) \subset P,
\end{equation*}
one concludes that $t + v_i \in P$. Thus $\| t + v_i \| \leq \| t \|$. On the other hand, $\| t + v_i \| \geq \| t \|$ by definition of $t$.
Hence
\begin{equation}\label{eq:10}
  \| t + v_i \| = \| t \|
\end{equation}

Consider the sphere $S$ of radius $\| t \|$ centered at $-t$. By~\eqref{eq:10},
\begin{equation*}
  v_i \in S \quad \text{for every $i \in \{ 1, 2, \ldots, 2^n \}$.}
\end{equation*}
On the other hand, each $v \in \Lambda$ satisfies $\| t + v \| \geq \| t \|$, so $S$ is an empty sphere, and
\begin{equation*}
\Pi = \conv (S \cap \Lambda)
\end{equation*}
is a Delaunay cell with at least $2^n$ vertices. From 
Propositions~\ref{prop:dela} and~\ref{prop:dagr} one concludes that 
$\Pi$ is a $d$-dimensional rectangular box,
hence so is $P$.
\end{proof}

\section*{Acknowledgements}

The author is indebted to Barak Weiss and Mathueu Dutour Sikiri\'c for fruitful discussions on the subject. The author is also indebted to an
anonymous referee for useful suggestions regarding the text.

\end{document}